\documentclass[12pt]{amsart}
\usepackage[margin=1.2in,marginparwidth=60pt]{geometry}
\usepackage{graphicx,amsmath,amsfonts,amssymb,color,amsthm}
\usepackage[all]{xy}

\renewcommand{\H}{\mathcal{H}}
\newcommand{\Q}{\mathbb{Q}}
\newcommand{\Z}{\mathbb{Z}}
\newcommand{\N}{\mathbb{N}}
\newcommand{\C}{\mathbb{C}}
\newcommand{\R}{\mathbb{R}}
\newcommand{\SL}{\mathrm{SL}}

\newcommand{\boV}{\mathcal{V}}
\newcommand{\Mod}{\mathrm{Mod}}
\newcommand{\PU}{\mathrm{PU}}
\newcommand{\SU}{\mathrm{SU}}

\renewcommand{\O}{\mathrm{O}}

\DeclareMathOperator{\Tr}{Tr}

\DeclareMathOperator{\Hom}{Hom}
\DeclareMathOperator{\Sign}{Sign}

\renewcommand{\epsilon}{\varepsilon}

\def\newsearrow{\raisebox{2pt}{{\scalebox{0.6}{$\searrow$}}}}
\def\newsmallsearrow{\raisebox{1pt}{{\scalebox{0.5}{$\searrow$}}}}

\newtheorem{Theorem}{Theorem}[section]
\newtheorem*{Theoremstar}{Theorem}

\newtheorem{Proposition}[Theorem]{Proposition}
\newtheorem{Remark}[Theorem]{Remark}
\newtheorem{Lemma}[Theorem]{Lemma}
\newtheorem{Conjecture}[Theorem]{Conjecture}
\newtheorem{Corollary}[Theorem]{Corollary}
\newtheorem*{Conjecturestar}{Conjecture}
\newtheorem{Question}[Theorem]{Question}
\title{Signatures in TQFT : Asymptotics and Modularity}
\author{Julien Marché}
\address{Département de mathématiques et applications, Ecole Normale Supérieure, Université PSL, CNRS, Sorbonne Université, 75005 Paris, France}
\email{julien.marche@ens.psl.eu}
\author{Gregor Masbaum}
\address{Sorbonne Université, Université Paris Cité, CNRS, IMJ-PRG, F-75005 Paris, France}
\email{gregor.masbaum@imj-prg.fr}
\date{February 21, 2026. First version: December 15, 2025}
\begin{document}

\begin{abstract}
  We study the signature $\sigma_g(\frac q p)$ of $\mathrm{SU}_2$-TQFT vector spaces associated to surfaces of genus $g$, as a function of the defining root of unity $\zeta=e^{i\pi q/p}$. We prove that $\frac{1}{p^2}\sigma_2(\frac{q}{p})$ converges to $\Lambda(\theta)=\frac{16}{\pi^3}\sum\limits_{n\ge 1, \textrm{ odd}}\frac{1}{n^3\sin(n\pi\theta)}$ when $\frac{q}{p}$ goes to an irrational number  $\theta\in [0,1]$ under certain conditions. We also observe that the function $\Lambda(\theta)$   is the boundary value of an Eichler integral of a level $2$ modular form of weight $4$, and use this to propose a conjectural transformation law for the signature function in genus~2 similar to the reciprocity formula for classical Dedekind sums. 
\end{abstract}
\maketitle

\tableofcontents

\section{Introduction}

Let $p$ be an odd positive integer and $\Q(\zeta)$ be the cyclotomic field generated by a primitive root of unity $\zeta$ of order $2p$. The SU$_2$-TQFT constructs for any surface $S_g$ of genus $g$ a finite dimensional $\Q(\zeta)$-vector space $\boV_p(S_g)$ endowed with an Hermitian form $h_g$ defined over $\Q(\zeta)$. These vector spaces are part of a modular functor, meaning that they support a projective unitary action of the corresponding mapping class groups  and satisfy a list of compatibilities when cutting the surfaces along simple closed curves.

The dimension of $\boV_p(S_g)$, given by the celebrated Verlinde formula, has received considerable attention since the nineties. On the other hand, taking $0<q<p$ two coprime odd integers, one can embed $\Q(\zeta)$ into $\C$ by sending $\zeta$ to $e^{i\pi q/p}$. The Hermitian form $h_g$ has a signature that we denote by  
$$\sigma_g(\textstyle{\frac{q}{p}})=\Sign\Big(\boV_p(S_g)\!\!\underset{\zeta=e^{i\pi q/p}}{\otimes}\!\!\C,\,h_g\Big)$$
and which has been very little studied. It was shown in \cite{DM} that for a fixed $\zeta=e^{i\pi q/p}$, these signatures have cut and paste properties very similar to those of the dimensions, which allows to encode them in a single Frobenius algebra $V_{\frac q p}$. Later it was proved in \cite{S2B} that $V_{\frac q p}$ is closely related to the hyperbolic geometry of the two-bridge knot $K(p,q)$. The purpose of this paper is to study signatures in a different direction, by analyzing the asymptotic behavior of $\sigma_{g}(\frac{q}{p})$ when $\frac{q}{p}\to \theta$ for a fixed $\theta\in [0,1]$.

If we take $q=1$, the Hermitian form $h_g$ becomes positive definite,
so that $\sigma_g(\frac 1 p)$ coincides with the dimension of $\boV_p(S_g)$.  The problem of computing the asymptotic behavior of the dimensions when $p\to \infty$ is well-known, and related to considerations of semi-classical analysis on character varieties $\mathcal{M}_g=\Hom(\pi_1(S_g),\SU_2)/\SU_2$. This led Witten  \cite{Witten} to show that for $g>1$,  
$$\lim_{p\to \infty}\frac{\dim \boV_p(S_g)}{p^{3g-3}}=\operatorname{Vol}(\mathcal{M}_g)=2(2\pi^2)^{1-g}\zeta(2g-2).$$

In genus $0$ and $1$, the hermitian form is positive definite for all values of $q$, so let us consider the first non trivial case $g=2$. In this case, one can write a simple formula for  $\sigma_2(\frac q p)$:
\begin{equation}\label{defsigma2}
  \sigma_2({\textstyle{\frac{q}{p}}})=\sum_{(j,k,\ell)\in \Delta_p}
 (-1)^{\lfloor {jq}/{p}\rfloor +\lfloor {kq}/{p}\rfloor+\lfloor {\ell q}/{p}\rfloor}
\end{equation}
 where the sum runs over the set $\Delta_p$ of triples of integers satisfying 
$$\begin{cases}0<j,k,\ell<p, \\
j<k+\ell,k<j+\ell,\ell<j+k,\\
j+k+\ell<2p,\textrm{ and } j+k+\ell\textrm{ odd}.
\end{cases}$$

Our main result is the following 

\begin{Theoremstar}
For almost all irrational $\theta\in [0,1]$, if we denote by $q_k/p_k$ the sequence of convergents of the continued fraction expansion of $\theta$, then one has
$$\lim_{k\to \infty}\frac{\sigma_2(q_k/p_k)}{p_k^2}=\frac{16}{\pi^3}\sum_{n\ge 1 \textrm{ odd}}\frac{1}{n^3\sin(n\pi\theta)}.$$
\end{Theoremstar}

Figure \ref{fig:sigma2}  suggests that the convergence is rapid: the graph of the map on the right hand side is shown in green whereas red dots represent the points $(q/p,\sigma_2(q/p)/p^2)$ for $0<q<p<31$. 

\begin{figure}[htbp]
\begin{center}
\includegraphics[width=10cm]{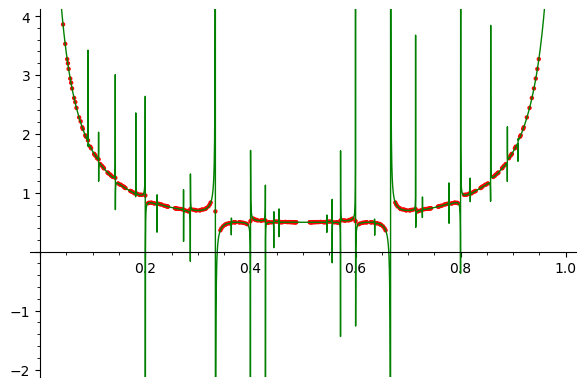}
\caption{Graph of the signature function $\sigma_2$}\label{fig:sigma2}
\end{center}
\end{figure}

The proof of this theorem involves several steps and occupies Sections~\ref{genre2} and~\ref{sec2}. We first apply in Section~\ref{genre2} a finite Fourier transform in Equation \eqref{defsigma2} to transform this expression into a sum of evaluations of a rational function at roots of unity indexed by $\Delta_p$. We then apply Brion's theorem to get an explicit trigonometric sum computing $\sigma_2(q/p)$; this step  presents miraculous cancellations for which we do not have any conceptual explanation. 

Then in Section~\ref{sec2}, we provide the asymptotic analysis needed to prove the theorem. Here we use a simplification of our trigonometric formula which was pointed out to us by Y. Murakami \cite{YMu}. This simplification makes the asymptotic analysis much easier than in our original proof.

In higher genus one can also write down a formula for
  $\sigma_g(q/p)$ as the sum of signs indexed by lattice points inside some polytope (see \cite[Remark 4.12]{BHMV}~\footnote{Our $\boV_p$ corresponds to the $V_{2p}$ of \cite{BHMV}.}.) But the signs are much more complicated than in Formula \eqref{defsigma2} which prevents us from generalizing our theorem. We briefly discuss in Section~\ref{sec5} how the asymptotics might look like for general surfaces including marked points, and illustrate this by some numerical experiments.

In the last section of the paper, we discuss some modular properties of the signature. We observe that the function 
$$\Lambda(\theta)=\frac{16}{\pi^3}\sum_{n\ge 1\textrm{ odd}}\frac{1}{n^3\sin(n\pi\theta)}$$
is the boundary value of an Eichler integral of a modular form of weight $4$ for the level $2$ group $\Gamma(2)$. In addition to the obvious property $\Lambda(\theta+1)=-\Lambda(\theta)$, it satisfies the following transformation law:
\begin{equation}\label{Lambdamod}
\Lambda\Big(\frac{\theta}{2\theta+1}\Big)(2\theta+1)^2-\Lambda(\theta)=2\theta^2+2\theta+1.
\end{equation}
This functional equation suggests that a similar integral version holds for signatures. In the first version of this paper (of December 15, 2025), we made the following conjecture,  which we had checked by computer for all $0<q<p<100$. 
\begin{Conjecturestar}
For any odd coprime integers $0<q<p$ one has
\begin{equation}\label{conj-form} \sigma_2\Big(\frac{q}{2q+p}\Big)-\sigma_2\Big(\frac{q}{p}\Big)=2q^2+2qp+p^2-1~.
  \end{equation}
\end{Conjecturestar}
In early February 2026, we were informed by Y. Murakami \cite{YMu} that he has obtained a proof of this conjecture starting from our trigonometric formula for $\sigma_2(\textstyle{\frac{q}{p}})$.

Our conjectured formula \eqref{conj-form} is reminiscent of the reciprocity formula of Dedekind sums: in fact we will make this analogy more precise in Section~\ref{SectionS}. We prove that the signature at $\zeta=e^{i\pi q/p}$ of  the TQFT-vector space associated to the 4 times punctured sphere (with  colors $(p-1)/2$  at the marked points)  is actually equal to a 2-smoothed version, $S(q/p)$,  of the Dedekind sum $s(q,p)$. It is then classical that  this function $S(q/p)$ is the boundary value of a modular form (this time of weight $1/2$) and that it satisfies the following transformation law

$$S\Big(\frac{q}{2q+p}\Big)-S\Big(\frac{q}{p}\Big)=1.$$
We provide an elementary proof of these facts for the benefit of the non-specialist reader. This very satisfactory state of things where signatures have modular properties would hopefully generalize to higher genus as well.

To end this introduction, we would like to mention the parallel paper \cite{MY} in which the first author and S. Yoon compute the asymptotics of the function
$\sigma_g(q/p)$, but for sequences of the form $\frac{q_n}{p_n}=\frac{a+bn}{c+dn}$ where $\left(\begin{smallmatrix} a & b \\ c & d \end{smallmatrix}\right)\in \Gamma(2)$. This regime is somewhat orthogonal to the one described in the present paper. The techniques used and the results are also very different. A synthesis of the two papers has yet to be made.
\vskip 8pt

{\bf Acknowledgements:} It is our pleasure to thank Pierre Charollois for early discussions on many aspects of this paper, Gaëtan Chenevier for his expertise on modular forms, 
and Don Zagier for very helpful explanations about Eichler integrals and quantum modular forms. We also thank Yuya Murakami for sharing a draft of his paper with us in early February 2026.

\section{Signatures in TQFT 
  and signed Verlinde algebras}\label{frobenius}

In this section, we recall the general formula for the signature which is used in the proof of Formula \eqref{defsigma2}
for $\sigma_2(\frac q p)$ (proved in Lemma \ref{triplepsilon} below). The reader willing to take Formula \eqref{defsigma2} for granted may  skip this section at first reading.

Given an odd integer $p$ and a coprime odd integer $q$ satisfying $0<q<p$, we set $\zeta=e^{i\pi q/{p}}$. We define a semi-simple, commutative Frobenius algebra over the rationals $V_{\frac q p}$ in the following way. 

As a $\Q$-vector space, $V_{\frac q p}$ has dimension $p-1$, with basis $e_0,e_1, \ldots, e_{p-2}$. We endow $V_{\frac q p}$ with a non-degenerate bilinear symmetric form $\eta$ such that $\eta(e_j,e_k)=0$ for $j\ne k$ and 
\begin{equation}\label{etasign}\eta(e_j,e_j)=(-1)^j \epsilon_{j+1},\text{ where }\epsilon_j=(-1)^{\lfloor \frac{jq}{p}\rfloor}=\Sign([j]).\end{equation}
Here we have used the quantum integer $[j]=\frac{\zeta^j-\zeta^{-j}}{\zeta-\zeta^{-1}}$. In the sequel, we will also use the quantum factorial $[n]!=[1][2]\cdots[n]$.

Let $T_p$ be the set of triples $(j,k,\ell)$ of integers satisfying $0\le j,k,\ell\le p-2$, $j+k+\ell\le 2p-4$, $j+k+\ell$ even, $\ell \le j+k,j\le k+\ell,k\le j+\ell$. 
Notice that increasing all indices by $1$ gives a bijection from $T_p$ to the set $\Delta_p$ defined in the introduction. From now on, we follow the TQFT convention that index colors (corresponding to the basis vectors of $V_{\frac q p}$) from $0$ to $p-2$ rather than from $1$ to $p-1$.

We define a trilinear symmetric form $\omega$ on $V_{\frac q p}$ by setting 
$\omega(e_j,e_k,e_\ell)=0$ if $(j,k,\ell)\notin T_p$, and otherwise 
$$ \omega(e_j,e_k,e_\ell)=(-1)^{\frac{j+k+\ell}{2}}\Sign\Big(\frac{[\frac{j+k+\ell}{2}+1]![\frac{j+k-\ell}{2}]![\frac{j+\ell-k}{2}]![\frac{\ell+k-j}{2}]!}{[j]![k]![\ell]!}\Big)$$

It was proven in \cite{DM} that the product $\cdot$ implicitly defined by the formula $\omega(x,y,z)=\eta(x\cdot y,z)$ defines a structure of commutative, associative, semi-simple Frobenius algebra on $V$, with unit $e_0$ and co-unit $\epsilon:V_{\frac q p}\to \Q$ defined by $\epsilon(e_0)=1$ and $\epsilon(e_j)=0$ for $j>0$. 

This Frobenius algebra was designed for computing signatures of TQFT. Recall that the Witten-Reshetikhin-Turaev SU$_2$-TQFT provides a compatible family of projective unitary representations $$\rho_p^\lambda:\Mod(S_{g,n})\to \PU(\boV_p(S_{g,n},\lambda))$$
In this formula $S_{g,n}$ denotes a compact oriented topological surface of genus $g$ with $n$ marked points and $\lambda=(\lambda_1,\ldots,\lambda_n)$ is a coloring of these points, that is,  a collection of $n$ integers in $\{0,1, \ldots, p-2\}$. The group $\Mod(S_{g,n})$ is the usual mapping class group fixing the marked points and $\boV_p(S_{g,n},\lambda)$ is a finite dimensional vector space over the cyclotomic field $K$ of order $2p$.  It is endowed with an invariant Hermitian form $h_g$ defined over $K$. Once $K$ is embedded in $\C$ by choosing the  defining root to be $\zeta=e^{i\pi q/p}$, the Hermitian form on $\boV_p(S_{g,n},\lambda)\otimes_K\C$ has signature $\sigma_{g,n}(\frac q p;\lambda)$ which strongly depends on the choice of $q$. Notice that as a vector space over $K$, $\boV_p(S_{g,n},\lambda)$ does not depend on the choice of $q$, which justifies the notation. 

Any Frobenius algebra has a special element $\Omega$ which can be defined by the formula $\Omega=\sum x_iy_i$ where $\eta^{-1}=\sum  x_i\otimes y_i \in V\otimes V$. As $V$ is semi-simple, $\Omega$ is invertible and satisfies $\epsilon(x)=\Tr_{V/\Q}(\Omega^{-1}x)$ for any $x\in V$. 

\begin{Theorem}[\cite{DM}]\label{DM1}
One has for any $g,n\ge 0$ and any $\lambda_1,\ldots,\lambda_n\in \{0,\ldots,p-2\}$:
$$\sigma_{g,n}({\textstyle\frac q p}; \lambda_1,\ldots,\lambda_n)=\epsilon(\Omega^g e_{\lambda_1}\cdots e_{\lambda_n})$$
where the right hand side is computed in the Frobenius algebra $V_{\frac q p}$.     
\end{Theorem}
In the sequel, when $n=0$, we will denote the signature simply by $\sigma_g(\frac q p)$. 

\begin{Remark}{\em Just like $V_{\frac q p}$, the $K$-vector space $\boV_p(S_1)$ has dimension $p-1$ and is in a natural way a Frobenius algebra known as the Verlinde algebra of the TQFT. This algebra only depends on $p$, not on $q$. The Frobenius algebra $V_{\frac q p}$ (which does depend on $q$) may be thought of as a signed Verlinde algebra, because the structure constants of the multiplication in $V_{\frac q p}$ coincide  up to sign with those of the multiplication in $\boV_p(S_1)$ in the standard bases. For $q=1$,  $\boV_p(S_1)$ is isomorphic to $V_{\frac 1 p}\otimes K$ as a Frobenius algebra,
    hence $\sigma_{g,n}({\textstyle\frac 1 p}; \lambda)=\dim \boV_p(S_{g,n},\lambda)$ for all $g,n\geq 0$.
 }  \end{Remark}
In genus zero $\boV_p(S_0)$ is the ground field, with the canonical Hermitian form, and  in genus $1$ the Hermitian form on $\boV_p(S_1)$  is positive definite for any choice of $q$.
  So $\sigma_0(\frac q p)=1$ and $\sigma_1(\frac q p)=p-1$ don't depend on $q$. Thus, for surfaces without marked points   interesting things happen only for $g\ge 2$.

Let us illustrate the Frobenius algebra formalism by giving the proof of Formula \eqref{defsigma2} for $\sigma_2(\frac q p)$ mentioned in the introduction. After the shift of indices from $\Delta_p$ to $T_p$, Formula \eqref{defsigma2} reads as follows.

\begin{Lemma}\label{triplepsilon}
$$\sigma_2({\textstyle\frac q p})=\sum_{(j,k,\ell)\in T_p} \epsilon_{j+1}\epsilon_{k+1}\epsilon_{\ell+1}.$$
\end{Lemma}

\begin{proof}
The trilinear form $\omega:V^3\to\Q$ can be viewed in the standard equivalence between Frobenius algebras and (1+1)-dimensional TQFTs, as the image of a pair of pants, interpreted as a cobordism from $S^1\amalg S^1\amalg S^1$ to $\emptyset$. Contracting two copies of it using the bilinear form $\eta$ corresponds topologically to gluing two pairs of pants along their boundaries, hence to a genus 2 surface. In formulas, this gives (using the orthogonality of the basis):
$$ \sigma_2({\textstyle\frac q p})=\sum_{j,k,\ell} \frac{\omega(e_j,e_k,e_\ell)^2}{\eta(e_j,e_j)\eta(e_k,e_k)\eta(e_\ell,e_\ell)}.$$
The lemma follows from the fact that $\omega(e_j,e_k,e_\ell)=\pm 1$ if $(j,k,\ell)\in T_p$ and $0$ otherwise, together with the observation that $(-1)^{j+k+\ell}=1$ for $(j,k,\ell)\in T_p$. 
\end{proof}

\begin{Remark}\label{23}
{\em By definition, the signature depends on $\zeta=e^{i\pi q/{p}}$, that is,  on the choice of $q$ modulo $2p$, not just modulo $p$.  But $q$ and $2p-q$ give the same signature, because the formulas only depend on the signs of the quantum integers $[j]$, and these signs are preserved  by the transformation $q\mapsto 2p-q$ or equivalently
$\zeta\mapsto \zeta^{-1}$. Thus we may assume $0<q<p$.
} \end{Remark}
\begin{Remark}
  {\em
The $\mathrm{SU}_2$-TQFT $\boV_p$ corresponds in Conformal Field Theory to the Wess-Zumino-Witten model at odd level $\ell=p-2$. Similar $\mathrm{SU}_2$-TQFTs exist at  even levels, that is, when $p$ is even. Their signature functions can again be computed using the Frobenius algebra formalism of Theorem~\ref{DM1}.  But we won't consider even level $\mathrm{SU}_2$-TQFTs in this paper.
}\end{Remark}

\section{A trigonometric formula for the signature in genus 2}\label{genre2}
Let us now start our investigation of the asymptotic behavior of the signature function. The first result of this paper is the following trigonometric formula for the signature in genus $2$.

\begin{Theorem}\label{sigma2}
  For $0<q<p$ coprime odd integers, we have
\begin{equation}\label{f1sigma2} \sigma_2({{\textstyle\frac q p}})= \frac {1-p^2}{6p^2} +
  \frac 1 {4p^2}\sum_{n=1,\, \mathrm{ odd}}^{p-2}\frac{f(n,p,q)}{\sin^3(\frac{n\pi}{2p})\sin^2(\frac{qn\pi}{p})}
\end{equation}
where 
\begin{eqnarray*} f(n,p,q)&=(3p-3)\sin(\frac{(2q-1)n\pi}{2p})+(p+1)\sin(\frac{(2q-3)n\pi}{2p})\\
&+(p-1)\sin(\frac{(2q+3)n\pi}{2p})+(3p+3)\sin(\frac{(2q+1)n\pi}{2p}).
\end{eqnarray*}
\end{Theorem}

We thank Y. Murakami for pointing out the following simplification of this formula:
  \begin{Corollary}[{\cite[Proposition~7.1]{YMu}}]\label{Cor-M}
    For $0<q<p$ coprime odd integers, we have
  \begin{equation}\label{f3sigma2}
    \sigma_2({{\textstyle\frac q p}})= \frac 2 p \sum_{n=1,\, \mathrm{ odd}}^{p-2} \frac {\cot^3 (\pi n /2p) } {\sin (\pi n q/p)}
    \end{equation}
  \end{Corollary}
 \begin{proof} Applying standard addition formulas for the sine and cosine functions in $f(n,p,q)$ and regrouping terms,  the expression \eqref{f1sigma2} can be rewritten as
    \begin{equation}\label{f2sigma2}
      \sigma_2({{\textstyle\frac q p}})= \frac {1-p^2}{6p^2}
      + 
       \frac 2 {p^2} \sum_{n=1,\, \mathrm{ odd}}^{p-2} \frac {\cos(\pi q n/p)} {\sin^2 (\pi q n/p)}
       +  \frac 2 p \sum_{n=1,\, \mathrm{ odd}}^{p-2}
       \frac {\cot^3 (\pi n /2p) } {\sin (\pi n q/p)} ~.
        \end{equation} But the first term and the first sum in \eqref{f2sigma2} cancel thanks to  the following trigonometric identity:
    \begin{Lemma}[{\cite[Equation~(7.1)]{YMu}}] For any odd $p>0$, one has 
        \begin{equation}\nonumber \sum_{m=1,\, \mathrm{ odd}}^{p-2} \frac {\cos(\pi m/p)} {\sin^2 (\pi m/p)} = \frac {p^2-1}{12}~.
          \end{equation}
\end{Lemma}
\end{proof}

\begin{proof}[Proof of Theorem~\ref{sigma2}]
By a finite Fourier transform over $\Z/2p\Z$, we can write for $m\in\Z$ :
\begin{equation}\label{FFT}
  (-1)^{\lfloor m/p\rfloor}=
  \frac 1 p
  \sum_{n=1,\text{ odd}}^{2p-1}\frac{\sin(\pi n(2m+1)/2p)}{\sin(\pi n/2p)}
  =\frac 1 p  \sum_{n=1,\text{ odd}}^{2p-1} \frac {t^{2m+1}-t^{-2m-1}}{t-t^{-1}}
    \Big|_{t=e^{\frac {i\pi n}{2p}}~.
      }
\end{equation}
In the sequel, we set $\O_p=\{1,3,\ldots,2p-1\}$ and $\O_p^+=\{1,3,\ldots,p\}$. Using Formula \eqref{FFT} for each sign $\epsilon_{j+1}=(-1)^{\lfloor q(j+1)/p\rfloor}$ and plugging it into the formula of Lemma \ref{triplepsilon}, we get 
\begin{equation}\nonumber\sigma_2({\textstyle{\frac q p}})\,\,
  =\sum_{(j,k,\ell)\in T_p} \epsilon_{j+1}\epsilon_{k+1}\epsilon_{\ell+1}
  =\frac1 {p^3} \sum_{(a,b,c)\in \O_p^3}\psi(a,b,c)\end{equation} 
where
\begin{equation}
\psi(a,b,c)=\!\!\!\!\sum\limits_{(j,k,l)\in T_p}\!\!\!\!
\frac{(x^{2q(j+1)+1}-x^{-2q(j+1)-1})(y^{2q(k+1)+1}-y^{-2q(k+1)-1})(z^{2q(\ell+1)+1}-z^{-2q(\ell+1)-1})}{(x-x^{-1})(y-y^{-1})(z-z^{-1})}\nonumber
\end{equation}
with $x=e^{\frac{i\pi a}{2p}}$,  $y=e^{\frac{i\pi b}{2p}}$ and $z=e^{\frac{i\pi c}{2p}}$.

The function $\psi$ is clearly symmetric in its three variables and satisfies also $\psi(2p-a,b,c)=\psi(a,b,c)$. We can hence rewrite $\sigma_2(\frac q p)$  as 
\begin{equation}\label{eq4n}\sigma_2({\textstyle{\frac q p}})=
  \frac1 {p^3}
  \sum_{a,b,c\in \O_p^+} 2^{m_a+m_b+m_c}\psi(a,b,c)\end{equation}
where $m_a=1$ if $a\ne p$ and $m_p=0$.

In order to compute $\psi(a,b,c)$, we rewrite it as follows:
\begin{equation} \psi(a,b,c)= \sum\limits_{\epsilon,\mu,\nu=\pm 1}\epsilon\mu\nu \frac{ x^{\epsilon}y^{\mu}z^{\nu}F(x^{2q\epsilon},y^{2q\mu},z^{2q\nu})}{(x-x^{-1})(y-y^{-1})(z-z^{-1})}\label{eq5n}
\end{equation} where $F$ is the polynomial defined by $F(X,Y,Z)=\sum_{(j,k,l)\in T_p}X^{j+1}Y^{k+1}Z^{\ell +1}$.

The computation of $F(x^{2q\epsilon},y^{2q\mu},z^{2q\nu})$ is simplified by the following observation:  we have $$F(X,Y,Z)=X^pY^p F(X^{-1}, Y^{-1},Z)~,$$ since the involution $(j,k,\ell ) \mapsto (p-2-j,p-2-k,\ell)$ preserves the set $T_p$. Therefore, since $F$ is symmetric in its three variables,  whenever $F(X,Y,Z)$ is evaluated at numbers $X=x^{2q}$, $Y=y^{2q}$, and $Z=z^{2q}$
satisfying $X^p=Y^p=Z^p=-1$, then for  any $\epsilon,\mu,\nu=\pm 1$, we have that $F(X^\epsilon, Y^\mu,Z^\nu)$ is equal to either $F(X,Y,Z)$ (when $\epsilon\mu\nu =1$) or to  $F(X^{-1}, Y^{-1},Z^{-1})= \overline {F(X,Y,Z)}$ (when $\epsilon\mu\nu =-1$).
\begin{Remark}\label{rk33}{\em For later use, we note the following special case: when $F(x^{2q},y^{2q},z^{2q})$ happens to be real, then $F(x^{2q\epsilon},y^{2q\mu},z^{2q\nu})$ 
    does not depend at all on $\epsilon,\mu$, or $\nu$, and formula \eqref{eq5n} for $\psi(a,b,c)$ simply becomes $\psi(a,b,c)=F(x^{2q},y^{2q},z^{2q})$.}
    \end{Remark}

We are now ready to compute $\psi(a,b,c)$. We will need to consider different cases depending on the number of distinct values taken by $a,b,c$.

\vspace{8pt}
{\bf I. The case when $a,b,c\in O_p^+$ are pairwise distinct.} We claim that  $\psi(a,b,c)=0$ in this case. To see this, it suffices to show that $F(x^{2q},y^{2q},z^{2q})=0$, where $x=e^{\frac{i\pi a}{2p}}$,  $y=e^{\frac{i\pi b}{2p}}$ and $z=e^{\frac{i\pi c}{2p}}$ are as above. We prove this  by evaluating  $F(x^{2q},y^{2q},z^{2q})$ using Brion's formula \cite{Br} (see \cite{Brion} for an elementary treatment), as follows.
 We can view $T_p$ as the intersection $\Lambda\cap \Delta$ where $\Delta$ is a simplex whose vertices are $(0,0,0)$, $(p-2,p-2,0)$, $(p-2,0,p-2)$ and $(0,p-2,p-2)$ and $\Lambda$ is the lattice of triples in $\Z^3$ with even sum. As $\Lambda$ is abstractly isomorphic to $\Z^3$, we can apply  Brion's formula  to write the polynomial $F$ as   $ F(X,Y,Z)=XYZ \, {\mathcal R}_p(X,Y,Z)$ where  ${\mathcal R}_p(X,Y,Z)$ is the following sum of four rational functions corresponding to the four vertices of $\Delta$:
\begin{multline*} {\mathcal R}_p(X,Y,Z)=\frac{1}{(1-XY)(1-XZ)(1-YZ)}+\frac{X^pY^p}{(XY-1)(X-Z)(Y-Z)}\\
+\frac{Y^pZ^p}{(YZ-1)(Y-X)(Z-X)}+\frac{X^pZ^p}{(XZ-1)(X-Y)(Z-Y)}
\end{multline*}
When $a,b,c\in O_p^+$ are pairwise distinct, we can use this rational function ${\mathcal R}_p$ to compute $F(x^{2q},y^{2q},z^{2q})$, as then all denominators are non-zero:  the numbers  $X=x^{2q}$, $Y=y^{2q}$, and $Z=z^{2q} $ are pairwise distinct and satisfy $XY\ne 1,XZ\ne 1,YZ\ne 1$. Furthermore, since these numbers satisfy the relations $X^p=Y^p=Z^p=-1$, we may also use the rational function ${\mathcal R}$ obtained from ${\mathcal R}_p$ by setting the three numerators $X^pY^p$, $Y^pZ^p$, $X^pZ^p$, equal to $1$. (Note that $p$ no longer appears in ${\mathcal R}$, justifying the notation.)  Writing now ${\mathcal R}$ with a common denominator, one finds by a tedious but straightforward computation that ${\mathcal R}=0$ as a rational function. Thus $F(x^{2q},y^{2q},z^{2q}) = x^{2q}y^{2q}z^{2q} \, {\mathcal R}(x^{2q},y^{2q},z^{2q})=0$, as asserted.

\vspace{8pt}

{\bf II. The case when $a,b,c\in O_p^+$ take precisely two distinct values.} We claim that the terms $\psi(a,b,c)$ where $\#\{a,b,c\}=2$, when taken all together, contribute zero to $\sigma_2(\frac q p)$. In view of \eqref{eq4n}, this will follow from the following two identities:
\begin{align}
  \psi(a,a,c)&= - \psi(c,c,a) &\mathrm{if}& \ a,c \ne p\  \mathrm{and}\ a\ne c \label{eq6n}\\
\psi(p,p,c)&= - 2 \psi(c,c,p) &\mathrm{if}&\ p\ne c
\label{eq7n}
\end{align}

To prove these identities, let us compute $\psi(a,a,c)$ for $a\ne c$. As before, we set $x=e^{\frac{i\pi a}{2p}}$, $z=e^{\frac{i\pi c}{2p}}$. We need to evaluate $F(X,X,Z)$ at $X=x^{2q}$, $Z=z^{2q}$. A direct computation gives
\begin{equation}\label{eq8n}
  F(X,X,Z)= \sum\limits_{(j,k,l)\in T_p}X^{j+k+2}Z^{\ell+1}= \sum_{\ell=1}^{p-1} \ell Z^l (X^{\ell +1} +  X^{\ell +3} + \ldots + X^{2p-1-\ell})~.
  \end{equation}
If $a\ne p$, then $X^2\ne 1$ and we can sum the terms in $X$ to get a rational function. The computation is tedious and the answer somewhat complicated. But after some manipulation using $X^p=Z^p=-1$, $X\ne Z$ and $XZ\ne 1$, we get
\begin{equation}\label{eq9n}
  F(X,X,Z)= \frac{pXZ}{(1-XZ)(X-Z)}~.
  \end{equation}
  We take two things from this: first, that $F(X^{-1},X^{-1},Z^{-1})= F(X,X,Z)$ (in other words, $F(X,X,Z)$ is real),  so that we have
  $\psi(a,a,c)=F(X,X,Z)$
  (see Remark~\ref{rk33}.)  And second, that $F(X,X,Z)=-F(Z,Z,X)$ so that $\psi(a,a,c)= - \psi(c,c,a)$. This proves the identity \eqref{eq6n}.

  If $a=p$, we need to proceed differently as now $X=-1$, so that Eq.~\eqref{eq8n} becomes
  $$F(-1,-1,Z)= - \sum_{\ell=1}^{p-1} \ell (p-\ell) (-Z)^\ell~. $$ But since $c\ne a=p$, we have $Z\ne -1$ so that we can now express this as a rational function in $Z$. Again after some manipulation using $Z^p=-1$, we get
  $$F(-1,-1,Z)=\frac {2pZ}{(1+Z)^2}~.$$ Again, this is real, so that we have
  $\psi(p,p,c)=F(-1,-1,Z)$.
  On the other hand, we can compute $\psi(c,c,p)$ using Eq.~\eqref{eq9n} to get
  $$ \psi(c,c,p)=F(Z,Z,-1) =\frac {-pZ} {(1+Z)^2}~.$$

  Thus $\psi(p,p,c)= - 2\psi(c,c,p)$, proving the identity \eqref{eq7n}. This concludes the discussion of Case II.

 We are now left with the equality
 \begin{equation}\label{factor8}\sigma_2({\textstyle\frac q p}) =
   \frac1 {p^3}
   \Bigl(\psi(p,p,p) + 2^3 \sum_{n\in \O_p^+, n\ne p}\psi(n,n,n)\Bigr).\end{equation}

{\bf III. The case when $a=b=c\in\O_p^+$.} We claim that  $\psi(p,p,p)$ gives rise to the first term in Theorem~\ref{sigma2}, while $\psi(n,n,n)$ for $n\ne p$ gives rise to the summand involving $f(n,p,q)$. The proof goes as follows. As before, we set $x=e^{\frac{i\pi n}{2p}}$ and  $X=x^{2q}$. A direct count gives 
\begin{equation}\label{eq10n}
  F(X,X,X)=\sum\limits_{(j,k,l)\in T_p}X^{j+k+l+3}=\sum\limits_{\ell=0}^{p-2}\frac{(\ell+1)(\ell+2)}{2}X^{2\ell+3}~.
\end{equation}
If $n=p$ then $X=-1$ and $F(-1,-1,-1)=-\binom{p+1}{3}$, hence  $$\textstyle{\frac 1 {p^3} \psi(p,p,p)=\frac 1 {p^3} F(-1,-1,-1)=  \frac{1-p^2}{6p^2}}~,$$ as asserted.

Let us now assume $n\ne p$. Then $X^2\ne 1$ and we can sum the R.H.S. of \eqref{eq10n} to get a rational function in $X$. Again the we skip the tedious computation. The result is
\begin{eqnarray*} F(X,X,X)&=&\frac{X^{2p}((p^2-p) X^2+2-2p^2+(p^2+p)X^{-2}))-2}{2(X-X^{-1})^3}\\
  &=& \frac {p\, \varphi(X)} {2(X-X^{-1})^2} ~,
\end{eqnarray*}
where $\varphi(X)=(p-1)X-(p+1)X^{-1}$. (We have used $X^p=-1$ in the last step.)

Note that  $\varphi(X)\ne\varphi(X^{-1})$ so that this time $F(X,X,X)$ is not real. Eq.~\eqref{eq5n} gives 
\begin{align}\nonumber
  \psi(n,n,n)&=\frac {(x^3+3x^{-1})F(X,X,X) -
    (3x+x^{-3})F(X^{-1},X^{-1},X^{-1})}{(x-x^{-1})^3}\\
             &=\frac {(x^3+3x^{-1})p\,\varphi(X) - (3x+x^{-3}) p \,\varphi(X^{-1})}{2 (x-x^{-1})^3 (X-X^{-1})^2}\label{psi-nnn}
\end{align}

Remembering now that $X=x^{2q}$, this becomes
$\psi(n,n,n)=pN/D$ where the denominator is 
$$D=2 (x-x^{-1})^3 (x^{2q}-x^{-2q})^2= {2(2i)^5\sin^3(\frac{n\pi}{2p})\sin^2(\frac{qn\pi}{p})}$$ and the numerator is
$pN$ with 
\begin{multline*} N=(p-1)(x^{2q+3}-x^{-2q-3})+(3p-3)(x^{2q-1}-x^{-2q-1})\\
  +(3p+3)(x^{2q+1}-x^{-2q-1})+(p+1)(x^{2q-3}-x^{-2q+3})
  = 2if(n,p,q)
\end{multline*}
where $f(n,p,q)$ is as defined in Theorem~\ref{sigma2}. Plugging this into Eq.~\eqref{factor8} proves the theorem.
\end{proof}

  \begin{Remark}{\em Murakami's rewriting \eqref{f2sigma2} of our trigonometric formula can be obtained directly from the above by writing  $\varphi(X) =  -(X+X^{-1}) + p(X-X^{-1})$ and rewriting the expression \eqref{psi-nnn}  as follows:
      \begin{align*}
        \psi(n,n,n)&=\frac { -(x-x^{-1})^3 p(X+X^{-1}) + (x+x^{-1})^3 p^2 (X-X^{-1})}{2 (x-x^{-1})^3 (X-X^{-1})^2}\\
        &= -\frac p 2  \, \frac {X+X^{-1}} {(X-X^{-1})^2}  + \frac {p^2} 2 \, \frac {(x+x^{-1})^3} {(x-x^{-1})^3} \, \frac 1 {X-X^{-1}}  ~.
      \end{align*}
      Remembering that $x=e^{\frac{i\pi n}{2p}}$ and  $X=x^{2q}$ and plugging this  into Eq.~\eqref{factor8} proves Formula \eqref{f2sigma2}.
      }\end{Remark}

\section{Asymptotic behavior of the signature in genus $2$}\label{sec2}

Here is the main result of this paper.
\begin{Theorem}\label{asympto}
  For almost all irrational $\theta \in [0,1]$, 
  if we write
  $q_k/p_k$
  for the convergents of the regular continued fraction expansion of $\theta$, we have :
  $$ \lim_{k\rightarrow \infty}\frac{1}{p_k^2}
  \sigma_2\Bigl(\frac{q_k}{p_k}\Bigr)
  =\frac{16}{\pi^3}\sum_{n\ge 1, \mathrm{odd}}\frac{1}{n^3\sin(n\pi\theta)}.$$
  \end{Theorem}

We remark that for almost all irrational $\theta$, the convergents $\frac {q_k}{p_k}$ have the property that  both $q_k$ and
$p_k$ are odd for infinitely many $k$ (see Remark~\ref{rem36}).
But in Theorem~\ref{asympto}, we don't
need to restrict to the convergents where both $q_k$ and
$p_k$ are odd.  We can take $\sigma_2(\frac q
p)$ to be the function defined by the expression in
Corollary~\ref{Cor-M}, so that $\sigma_2(\frac{q_k}{p_k})$ makes sense
also when one of $q_k$ or $p_k$ is even.
Our proof will show that with this
interpretation of $\sigma_2(\frac q
p)$ the theorem holds without parity restrictions on the convergents. 

We do not know the answer to the following question:
\begin{Question} Does Theorem~\ref{asympto} hold if the convergents  $\frac {q_k}{p_k}$ are replaced with an arbitrary sequence of rational numbers $\frac r s$ converging to $\theta$ ? Say with both $r$ and $s$ odd, so that  $\sigma_2(\frac r s)$ is indeed the
signature of a TQFT vector space in genus two?
\end{Question}

  Before proving Theorem~\ref{asympto}, let us discuss the convergence of the infinite sum on the right hand side. We recall that for any irrational number $\theta$, its irrationality exponent $\mu(\theta)$ is the supremum of  the set of $\mu\in \R$ such that the inequality 
$$0<|\theta-\frac{q}{p}|<\frac{1}{p^\mu}$$
is satisfied by an infinite number of coprime integer pairs $(p,q)$ with
$p>0$.
One always has  $\mu(\theta)\geq 2$, as follows from the continued fraction expansion (see Proposition~\ref{qpfacts}(iv) below.) Moreover it is well-known that almost all irrational numbers $\theta$ are {\em diophantine}, that is, have irrationality exponent $\mu(\theta)$ precisely equal to $2$.  In order to prove Theorem~\ref{asympto}, it will be sufficient to assume that $\mu(\theta)<3$. Assuming this from now on, take $0<\epsilon <3-\mu(\theta)$, then the inequality $|\theta-\frac{q}{p}|\ge {p^{-3+\epsilon}}$ holds for all but a finite number of pairs $(p,q)$. Observe that we do not need to impose the condition that $p$ and $q$ are coprime anymore.

Let us see how this implies convergence of the sum on the R.H.S. of Theorem~\ref{asympto}. 
Given $n\in \N$, one derives a
lower bound for
$\sin(\pi n\theta)$ in the following way. First, observe the inequality 
\begin{equation}\label{sin-ineq} |\sin (\pi x)| \geq 2 |x| \ \ \ \ \mathrm  {for \ \ all } \ \ |x| \leq \textstyle{\frac 1 2}~. 
\end{equation}  Now take $\ell\in \Z$ such that $|n\theta-\ell|\le\frac{1}{2}$: for such an $\ell$ one has $|\sin(n\pi\theta)|=|\sin(\pi(n\theta-\ell))|\ge 2n|\theta-\frac \ell n|$ by (\ref{sin-ineq}) and hence
\begin{equation}\label{sin2} |\sin(n\pi\theta)| \ge 2 {n^{-2+\epsilon}}
\end{equation}
for all but a finite number of $n$'s.
Plugging this inequality into the general term of the series on the R.H.S. of Theorem~\ref{asympto}  ensures its convergence.

We briefly review the properties of the convergents $q_k/p_k$ of $\theta$ that we will need. See {\em e.g.} Perron \cite{P} or Khinchin \cite{Kh1} for more details\footnote{Unfortunately, \cite{Kh1} writes $p_k$ for our  $q_k$ and {\em vice versa.}}. Each irrational $\theta\in {\mathbb R}$ has a unique regular continued fraction expansion $$\theta = [a_0; a_1, a_2, \ldots ] = a_0+\cfrac{1}{a_1+\cfrac{1}{a_2+\dots}} $$ where regular means that $a_0\in \Z$ and $a_i\in \Z_{\geq 1}$ for $i\geq 1$. Note that in our case $a_0=0$ since
$\theta \in [0,1]\setminus \Q$.
The convergents
$$\frac{q_k}{p_k} = [a_0; a_1, a_2, \ldots, a_k] =  a_0+\cfrac{1}{a_1+\cfrac{1}{\dots+ \frac{1}{a_k}}} $$ can also be defined recursively by
\begin{eqnarray} q_k&=a_k q_{k-1} + q_{k-2}, \ \ \ \ q_0=a_0,  \ \ \ \ q_{-1}=1 \label{rec5}\\
   p_k&=a_k p_{k-1} + p_{k-2}, \ \ \ \ p_0=1, \ \ \ \ p_{-1}=0 \label{rec6}
\end{eqnarray}
The following classical facts are well-known. 
\begin{Proposition}\label{qpfacts}
  For all irrational $\theta$, the convergents  $q_k/p_k$ of $\theta$ satisfy the following properties:
  \begin{enumerate} 
  \item[(i)] One has $1=p_0\leq p_1<p_2<\ldots$, {\em i.e.} the denominators $p_k$  are positive integers and moreover strictly increasing from $p_1$ onwards.
  \item[(ii)] For all $k\geq 1$, $q_k$ and $p_k$ are coprime.
  \item[(iii)] $\lim_{k\rightarrow\infty} q_k/p_k = \theta$.
  \item[(iv)]
    For all $k\geq 1$, we have
    \begin{equation}\nonumber 
      \Big|\theta           - \frac {q_k}{p_k} \Big|
      < \frac 1 {p_k^2}~.
      \end{equation}
    \end{enumerate}
  \end{Proposition}

      \begin{Remark}\label{rem36}{\em It is also known that for almost
all $\theta$, each positive integer occurs infinitely many times
(in fact, with positive
frequency, given by the Gauss-Kuzmin law)
as a coefficient $a_i$ of the regular continued
fraction expansion of $\theta$. In particular, infinitely many $a_i$
are odd. Using the recursion formulas (\ref{rec5}) and (\ref{rec6}),
it follows that for almost all $\theta$, both $q_k$ and $p_k$ are odd
for infinitely many $k$.  }\end{Remark}

\begin{proof}[Proof of Theorem~\ref{asympto}]
    We start from Murakami's simplification of our trigonometric formula for  $\sigma_2(\frac{q}{p})$ in Corollary~\ref{Cor-M}. Thus we must show
  \begin{equation}\label{eq4}
   \lim_{k\rightarrow \infty} 2 \sum_{n=1,\textrm{ odd}}^{p_k-2}\frac{\cos^3(\pi n/2p_k)}{p_k^3 \sin^3(\pi n/2p_k) \sin(\pi n q_k/p_k)}=\frac{16}{\pi^3}\sum_{n\ge 1, \textrm{odd}}\frac{1}{n^3\sin(n\pi\theta)}~.
    \end{equation}
Since $\lim {q_k}/{p_k}=\theta$ and
$\lim_{x\rightarrow 0}(x/ \sin x) =1$,   the limit as $k\rightarrow\infty$ of the $n$-th summand on the left of (\ref{eq4})
is precisely the $n$-th summand on the right. (Here we define the
$n$-th summand on the left to be zero if $n> p_k-2$.) Thus it suffices to show that when $\theta$ is irrational with $\mu(\theta)<3$, we may interchange the limit and the sum.

Let  $S_k$ denote the sum on the L.H.S. of \eqref{eq4} (before taking the limit as $k\rightarrow \infty$) and let $S^{(\theta)}_{k}$ the same sum but with $\sin(n\pi\theta)$ in the denominator in place of $\sin({q_kn\pi}/{p_k})$. Using (\ref{sin-ineq}), we see that the $n$-th summand of $S^{(\theta)}_{k}$ is bounded from above by $16/(n^3\sin(n\pi\theta))$ which, as we have seen, for $\mu( \theta)<3$  is the general term of an absolutely  convergent sum. Thus for $S^{(\theta)}_k$ we may interchange the limit and the sum, and we get
$$\lim_{k\rightarrow \infty} S^{(\theta)}_k= \frac{16}{\pi^3}\sum_{n\ge 1, \textrm{odd}}\frac{1}{n^3\sin(n\pi\theta)}~.$$
It remains to see that the difference $d_k=S_k-S^{(\theta)}_{k}$ goes to zero as $k\rightarrow \infty$. We proceed as follows. Note first that (\ref{sin2}) implies that for every
$0<\epsilon <3-\mu(\theta)$ there is $C_\epsilon>0$ so that
\begin{equation}\label{sinn-inequ} |\sin(n\pi\theta)| \ge C_\epsilon
{n^{-2+\epsilon}}
\end{equation} holds for all $n\geq 1$. We have 
      $$d_k= \  2  \sum_{n=1,\textrm{ odd}}^{p_k-2}\frac{\cos^3 (n\pi/2p_k)}{p_k^3\sin^3({n\pi}/{2p_k}) }\Biggl (\frac {1} {\sin({q_kn\pi}/{p_k})} - \frac {1} {\sin(n\pi\theta)}\Biggr)~.$$ We can bound the denominators from below using (\ref{sin-ineq}), (\ref{sinn-inequ}), and the trivial lower bound $|\sin({q_kn\pi}/{p_k})| \geq \sin (\pi/p_k) \geq 2/p_k$. Thus for $0<\epsilon <3-\mu(\theta)$ we have
\begin{equation}\label{eq6}
        |d_k| \leq \ 2 \sum_{n=1,\textrm{ odd}}^{p_k-2} \frac 1 {n^3} \frac {p_k}2  \frac {n^{2-\epsilon}}{ C_\epsilon} \Big|\sin (n\pi\theta) - \sin (\frac{q_kn\pi}{p_k})\Big|~.
      \end{equation}
      Now the addition formula $\sin\alpha - \sin \beta = 2 \sin \frac {\alpha-\beta} 2 \cos \frac {\alpha+\beta} 2$  gives
      \begin{eqnarray}\Big|\sin (n\pi\theta) - \sin \Bigl(\frac{q_kn\pi}{p_k}\Bigr)\Big|
        &\leq &2 \Big|\sin \Bigl(\frac {n\pi} 2 \Bigl(\theta- \frac{q_k}{p_k}\Bigr)\Bigr)\Big| \nonumber \\
        &\leq &n\pi \Big|\theta- \frac{q_k}{p_k}\Big| \nonumber\\
        &< &\frac {n\pi}{p_k^2} = \frac { n^{\epsilon}  n^{1-\epsilon} \pi}{p_k^{1+\epsilon}p_k^{1-\epsilon}}\label{eq7}\\
      &\leq &  \frac {n^\epsilon  \pi}{p_k^{1+\epsilon}}\label{eq8}
      \end{eqnarray} where the inequality in (\ref{eq7})
      follows from Prop.~\ref{qpfacts}(iv),
and the inequality (\ref{eq8}) follows from
$n\leq p_k-2$.  Plugging this into (\ref{eq6}) gives $$|d_k| \leq
 \frac {\pi C_\epsilon^{-1}} {p_k^\epsilon}\sum_{n=1}^{p_k}
\frac{1}{n}\leq \frac {\pi}{C_\epsilon}\frac{1+\ln(p_k)}{p_k^{\epsilon}}.$$ 
Thus $|d_k|\rightarrow 0$ as $k\rightarrow \infty$. This concludes the proof of Theorem~\ref{asympto}.
\end{proof}

\section{The higher genus case}\label{sec5}
In this section, we briefly discuss how the asymptotic behaviour of the genus $2$ signature might generalize to higher genus. Numerical experiments (see below) indicate that
  $\sigma_{g,n}(\frac{q}{p};\lambda)$ might grow like $p^{\max\{g,2g-2\}}$ when $\frac q p$ goes to an irrational $\theta\in[0,1]$. So we ask the following question.
\begin{Question}\label{51}
 Given $g,n\in \Z_{\geq 0}$ and $n$-tuples of integers  $\lambda=(\lambda_1,\ldots,\lambda_n)\in  \Z^n_{\geq 0}$,
 is there a function $F_{g;\lambda}:
 [0,1]
 \to \R$ defined almost everywhere satisfying the following:  
  For almost all irrational $\theta\in[0,1]$,
  if we write $\frac{q_k}{p_k}$ for the convergents of $\theta$, one has 
$$ \lim_{k\to\infty}\frac{\sigma_{g,n}(\frac{q_k}{p_k};\lambda)}{p_k^{\max\{g,2g-2\}}}=F_{g;\lambda}(\theta)$$
\end{Question}
Theorem \ref{asympto} shows that the answer is yes for $g=2$ and $n=0$. Let us
compare this to the asymptotics of the dimension
of the TQFT vector spaces
which 
is well-known: it has a semi-classical interpretation for which we refer to \cite[Section 3]{Witten}. Precisely:
$$ \lim_{p\to\infty}\frac{\dim \boV_p(S_g)}{p^{\max\{g,3g-3\}}}=\operatorname{Vol}(\mathcal{M}_g)=2(2\pi^2)^{1-g}\zeta(2g-2)$$
where the last equality holds only for $g>0$ and $\mathcal{M}_g$ is the character variety of representations of $\pi_1(S_g)$ into SU$_2$, endowed with its Liouville measure. We have no conceptual explanation why the order of growth $3g-3$
for the dimension
should be replaced by $2g-2$
for the signature.

We now observe that the question is easily answered affirmatively for $g=0$. Indeed, take any irrational $\theta$ and set $\zeta=e^{i\pi \theta}$. We can define an infinite dimensional
``almost Frobenius''
algebra $V_\theta$ as follows: as a vector space $V_\theta=\bigoplus_{n\ge 0} \Q e_n$. The form $\eta$ is such that the basis $e_n$ is orthogonal and satisfies 
$$\eta(e_n,e_n)=(-1)^n \Sign [n+1]=(-1)^n (-1)^{\lfloor (n+1)\theta\rfloor}.$$

Finally, the trilinear form $\omega$ and the product are defined as in Section \ref{frobenius}, putting $p=\infty$. The algebra $V_\theta$ is still commutative and associative, but it does not give a $(1+1)$-TQFT because $\Omega$ does not converge. However, the formulas for the signature in genus $0$, which do not involve $\Omega$, still make sense and we have 
$$\lim_{\frac{q}{p}\to\theta}\sigma_{0,n}(\textstyle{\frac{q}{p}};\lambda_1,\ldots,\lambda_n)=\epsilon(e_{\lambda_1}\cdots e_{\lambda_n}).$$
Thus the function  $F_{0;\lambda_1, \ldots, \lambda_n}(\theta) = \epsilon(e_{\lambda_1}\cdots e_{\lambda_n})$  answers Question~\ref{51} in the genus zero case.\footnote{
      Note that $\epsilon(e_{\lambda_1}\cdots e_{\lambda_n})$ depends on $\theta$, because the multiplication in $V_\theta$ depends on $\theta$.}

The ``almost Frobenius'' algebra $V_\theta$ is a kind of signed version of the algebra of regular class functions on SU$_2$, where $e_n$ corresponds to the character of the \mbox{$(n+1)$}-dimensional irreducible representation of SU$_2$. It would be interesting to have a geometric interpretation of it. 

It is also natural to try to extend the results of the previous section to all genus $g\ge 3$. Unfortunately, the formula for the signature is not as nice as in Lemma \ref{triplepsilon} and we do not know how to proceed for computing its asymptotic behavior. A numerical experiment for $g=3$ shown in Figure~\ref{genre3} seems to be compatible with an affirmative answer to the question but it is less convincing than in genus $2$.  

\begin{figure}[htbp]
\begin{center}
\includegraphics[width=10cm]{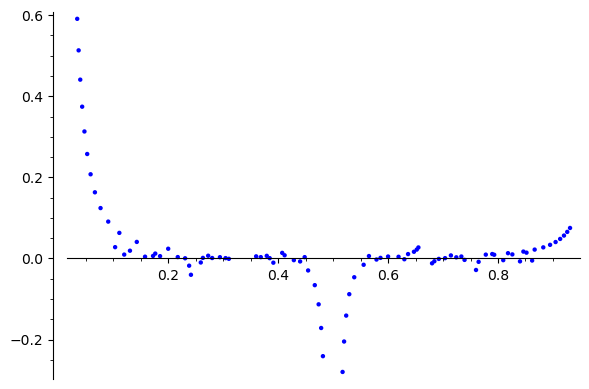}
\end{center}
\caption{The graph of the map
 $\frac{q}{p}\mapsto \sigma_3(\frac q p)/{p^4}$
  for $p\le 31$.}\label{genre3}
\end{figure}

As explained in Section \ref{frobenius}, the case $g=1$ and $n=0$ is not interesting as $\sigma_1(\frac{q}{p})=p-1$
for all $q$. However, interesting phenomena occur when adding marked points. For instance, given a positive integer $k$, it is not difficult to show the following formula: 
$$\sigma_1({\textstyle{\frac{q}{p}}};2k)=\sum_{n=k+1}^{p-1-k}\prod_{\ell=1}^k\epsilon_{n+\ell}\epsilon_{n-\ell}.$$
In Figure~\ref{fig3a} we show the graph of the maps $\frac{q}{p}\mapsto \frac{1}{p}\sigma_1(\frac{q}{p};2k)$ for $p\le 101$. 

\begin{figure}[htbp]
\begin{center}
\includegraphics[width=5cm]{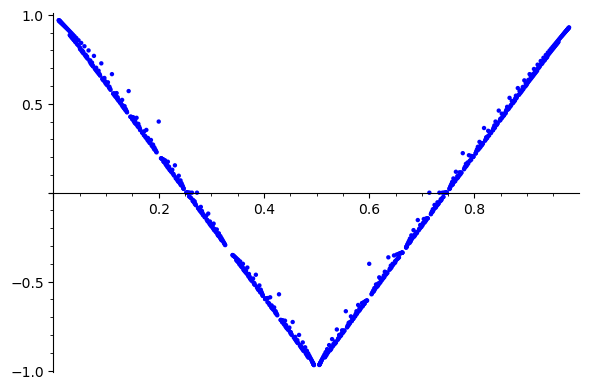}
\includegraphics[width=5cm]{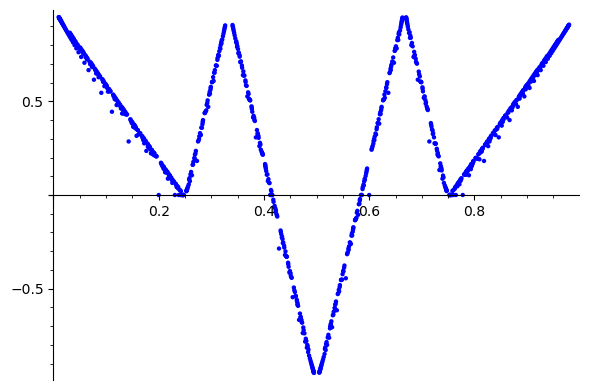}
\includegraphics[width=5cm]{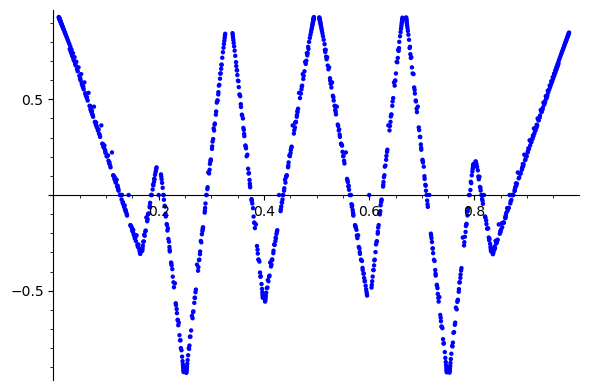}
\end{center}
\caption{Normalized signature of the punctured torus with colors
  $2k=2,4,6$.}
\label{fig3a}
\end{figure}

This case can probably be studied using the same techniques as in Sections \ref{genre2} and \ref{sec2} but we did not pursue in this direction.

\section{Modularity properties and transformation laws for the signature}

In this section, we discuss modular properties of the signature in two examples : the 4-punctured sphere and the closed surface of genus $2$.  
\subsection{A genus 0 example related to Dedekind sums}\label{SectionS}
Given $ 0<q<p$ coprime odd integers,
 we set
$\epsilon_n=(-1)^{\lfloor {nq}/{p}\rfloor}$
 as before, and define
 \begin{equation}\label{Sdef1}
   S\Bigl(\frac q p\Bigr)
   =\sum_{n=1,\, \rm{ odd}}^{p-1}\epsilon_n
  ~.\end{equation}

This integer has many interpretations in terms of signatures:
\begin{enumerate}
\item It is the signature of $\boV_p(S_{0,4}, \frac{p-1}{2},\frac{p-1}{2},\frac{p-1}{2},\frac{p-1}{2})$ at $\zeta=e^{i\pi q/p}$.
\item It is the signature of the bilinear form $\eta$ on the sub-algebra  of the Frobenius algebra $V_{\frac q p}$ generated by even colors. 
\item It is half the signature of the two-bridge knot $K(p,q)$, see \cite{S2B}.
\end{enumerate}

The second point is obvious, so let us explain the first one. Set $r=\frac{p-1}{2}$ and start from the formula 
$$
\sigma_0({\textstyle\frac q p}; r,r,r,r) =\epsilon(e_{r}^4)=\eta(e_r^2,e_r^2)$$
(see Theorem~\ref{DM1}.) The definition of the multiplication in $V_{\frac q p}$ gives  $e_r^2=\sum_{n} \frac{\omega(e_r,e_r,e_n)}{\eta(e_n,e_n)}e_n$.  The point now is that for $r=\frac{p-1}{2}$, one has that $(r,r,n)$ belongs to $T_p$ for all $n$ in $\{0,2,\ldots,p-3\}$. Hence we get $e_r^2=\sum_{n=0,\rm even}^{p-3}\pm e_n$. Since the $e_n$ are orthogonal, this implies $\eta(e_r^2,e_r^2)=\sum_{n=0,\rm even}^{p-3}\eta(e_n,e_n)$. The result now follows from the definition of $\eta$ (see Formula ~\eqref{etasign} in Section~\ref{frobenius}).
  
\vskip 8pt
 
  Let us now discuss modular properties of $S(\frac q p)$. For this it will be natural to allow $q$ to be negative\footnote{
    But notice that $S(\frac {-q} p) = - S(\frac q p)$ while the signature is unchanged when $q$ is replaced by $-q$, see Remark~\ref{23}. Thus our interpretation of $S(\frac q p)$ as a signature only holds for positive $q$.}  and $p$ (but not $q$) to be even. We redefine $S(\frac q p)$ for coprime integers $q,p$ with $p>0$ and $q$ odd, as follows: \begin{equation} \label{Sdef2}  S\Bigl(\frac q p\Bigr)=\frac 1 2 \sum_{n=1}^{p-1}\epsilon_n~,\end{equation} where $\epsilon_n=(-1)^{\lfloor {nq}/{p}\rfloor}$ as before. If $p$ is odd, this coincides with the original definition \eqref{Sdef1} because of the identity $\epsilon_{p-k}=\epsilon_k$.
 This identity also shows that when $p$ is even, $S(\frac q p)$ is not an integer, but a half-integer.

 \subsubsection{Relation to Dedekind sums}

Let us recall (one of) the definitions of the Dedekind sum, for coprime integers $p,q$ with $p>0$:
$$s(q,p)=\frac{1}{4p}\sum_{n=1}^{p-1}\cot\Bigl(\frac{\pi n}{p}\Bigr)\cot\Bigl(\frac{\pi n q}{p}\Bigr)~.$$

We give an explicit relation between the two sums, $S(\frac q p)$ and $s(q,p)$,
answering a question in \cite{S2B}.
In the terminology of number theorists, $S$ is a 2-smoothed version of $s$: 
\begin{Proposition}\label{dedekindtosign}
  For coprime integers $p,q$ with $p>0$ and $q$ odd,
    we have $$S({\textstyle\frac q p})=4s(q,2p)-2s(q,p).$$
  \end{Proposition}
  \begin{Remark}{\em
Note that $q$ must be odd for $s(q,2p)$ to make sense.
      }\end{Remark}

\begin{proof}
  From the finite Fourier transform formula \eqref{FFT} we get
  $$S\Bigl(\frac q p\Bigr)= \frac{1}{2p}\sum_{n=1}^{p-1}\sum_{k=1,\rm{odd}}^{2p-1}
  \frac{\sin(\frac{\pi k (2nq+1)}{2p})}{\sin(\frac{\pi k}{2p})}.$$
  By exchanging the sums and computing the geometric sum over $n$, we find
  $$S\Bigl(\frac q p\Bigr)=
  \frac{1}{2p}  \sum_{k=1,\rm odd}^{2p-1}
  \cot\Bigl(\frac{\pi k}{2p}\Bigr)\cot\Bigl(\frac{\pi k q}{2p}\Bigr).$$
Let us now cut the sum
$s(q,2p)=\frac{1}{8p}\sum_{n=1}^{2p-1}\cot(\frac{\pi
  n}{2p})\cot(\frac{\pi n q}{2p})$
into two parts depending on the parity of $n$: we get
$$s(q,2p)=\frac{1}{8p}\sum_{n=1}^{p-1}\cot(\frac{\pi n}{p})\cot(\frac{\pi n q}{p})+\frac{1}{8p}\sum_{n=1,\rm odd}^{2p-1}\cot(\frac{\pi n}{2p})\cot(\frac{\pi n q}{2p})$$
This yields the equality
$s(q,2p)=\frac{1}{2}s(q,p)+\frac{1}{4}S(\frac q p)$ from which the
result follows.
\end{proof}

\subsubsection{A modular interpretation for $S(\frac q p)$}

The Dedekind sums are known to have a chaotic behavior so that we cannot hope for a result in the spirit of Theorem \ref{asympto}. However, it is known that we can write $S(\frac q p)$ as an integral of a modular form. We give the details in this section as we will observe that the limiting series of Theorem \ref{asympto} has a similar interpretation. We believe that this similarity hides a general structure that deserves to be discovered. 

\begin{Proposition}\label{Smodular}
  Let $\H$ be the upper half-plane and $\eta:\H\to \C$ the Dedekind eta function,  given for $\tau \in \H$ by $\eta(\tau)=e^{i\pi\tau/12}\prod_{n=1}^{\infty}(1-e^{2i\pi n\tau})$.
Set $g(\tau)= \eta(\tau)^2/\eta(2\tau)$. 
Then for coprime integers $p,q$ with $p>0$ and $q$ odd,
we have
  $$S\Bigl(\frac q p\Bigr)= \frac{2}{\pi}\int_{\frac{q}{2p}}^{ i \infty}
  d\arg g(\tau)=-\frac{2}{\pi}\lim_{t\newsmallsearrow 0}\arg g\Bigl(\frac{q}{2p}+it\Bigr).$$
\end{Proposition}

Here the integration goes along the vertical line in $\H$ of real part $q/2p$, and $\arg g(\tau)$ stands for the function on $\H$ defined by setting  $\arg g(i\infty)= \arg 1 =0 $ and extending this to $\tau\in\H$ by continuity along any path from $i\infty$ to $\tau$.

\begin{proof}  
 This can be proved from the expression of $S(\frac q p)$ as  $2$-smoothed Dedekind sum as in Proposition \ref{dedekindtosign}, by adapting arguments given in \cite[Section 2.5]{DD}. Let us sketch for the benefit of the reader how the proposition can be proved by elementary means starting directly from the definition of $S(\frac q p)$ as the half-sum of signs $\epsilon_n$. We first set for $z=e^{2i\pi \tau} :$
  \begin{equation}\label{eq25} 
    F(z)=\sum_{n\ge 1} \arg\frac{1+z^n}{1-z^n}=\arg \prod_{n\ge 1}\frac{1+z^n}{1-z^n}=\arg \frac{\eta(2\tau)}{\eta(\tau)^2} = -\arg g(\tau)~.
    \end{equation}
Here, the first two occurrences of $\arg$ in (\ref{eq25}) are the principal determination of the argument for which each term in the sum belongs to
$]-\frac{\pi}{2},\frac{\pi}{2}[$, since the map $z\mapsto w=(1+z)/(1-z)$ sends the open disk $|z|<1$ to the half-plane $\mathrm{Re}(w)>0$. For later use we remark that this map sends $z=-1$ to $w=0$ and the upper ({\em resp.} lower) half of the circle $|z|=1$ to the upper ({\em resp.} lower) half of the imaginary axis, so that for $\zeta=e^{\pi i q/p}$ and $n\geq 1$ not a multiple of $p$, we have
\begin{equation} \label{eq26n}
  \arg \frac{1+\zeta^n}{1-\zeta^n}= \frac{\pi}{2}(-1)^{
 \lfloor {nq}/{p}\rfloor}
= \frac{\pi}{2}\epsilon_n~.
\end{equation}

Next, let us write $F=\sum_{k=0}^{p-1}F_k$ where
$$ F_k(z)=\sum_{n\geq 1, \ n  \equiv k \hspace {-7pt}\pmod p} \arg \frac{1+z^n}{1-z^n}.$$ Since $\zeta^p=-1$ is real, $F_0(\zeta e^{-t})$ vanishes for every $t>0$.  So its limit as $t\newsearrow 0$ is zero as well. Hence the proposition follows from \eqref{eq25} and the following lemma.
\end{proof}
\begin{Lemma}\label{lem63} For $k\in \{1,\ldots,p-1\}$ we have
  \begin{equation}\label{eq28n} \lim_{t\newsmallsearrow 0}F_k(\zeta e^{-t})=\frac \pi 4 (-1)^{
\lfloor{kq}/{p}\rfloor}= \frac{\pi}{4}\epsilon_k~.
\end{equation}
\end{Lemma}

\begin{Remark}{\em
    Note that by (\ref{eq26n})
 the limit as  $t \newsearrow 0$ of the summands in the series $F_k(\zeta e^{-t})$ is  alternatingly  equal to $\frac{\pi}{2}\epsilon_k$ or to $-\frac{\pi}{2}\epsilon_k$.
But in (\ref{eq28n}) we take the limit $t\newsearrow 0$ after first summing over $n$. Note that there is some subtle cancellation going on, as the limit is not zero, but $\frac{\pi}{4}\epsilon_k$.}
\end{Remark}

\begin{proof}[Proof of Lemma~\ref{lem63}.] For $t>0$, we pack the terms in the  series $F_k(\zeta e^{-t})$ two by two, giving
$$F_k(\zeta e^{-t})=\sum_{m\ge 1} \arg\frac{(1+\zeta^k e^{-(k+2pm)t})(1-\zeta^k e^{-(k+p+2pm)t})}{(1-\zeta^ke^{-(k+2pm)t})(1+\zeta^k e^{-(k+p+2pm)t})}$$
From the Taylor expansion $\log\frac{1+x}{1-x}=2\sum\limits_{\ell, \rm odd}\frac{x^\ell}{\ell}$ we get:

$$F_k(\zeta e^{-t})=2\operatorname{Im}\sum_{m\geq 1}\,\sum_{\ell,\rm odd} \frac{\zeta^{k\ell}}{\ell}\big(e^{-(k+2pm)\ell t}-e^{-(k+2pm+p)\ell t}\big).$$
Summing over $m$ this gives:
\begin{equation}\label{eq39m} F_k(\zeta e^{-t})=2\operatorname{Im}\sum_{\ell,\rm{ odd}} \frac{\zeta^{k\ell}}{\ell}e^{-k\ell t}\frac{e^{-2p\ell t}}{1-e^{-2p\ell t}}(1-e^{-p\ell t})=\operatorname{Im}\sum_{\ell,\rm odd} \frac{\zeta^{k\ell}}{\ell}e^{-k\ell t}\frac{2e^{-2p\ell t}}{1+e^{-p\ell t}}.
\end{equation}

\begin{Lemma}\label{lem64} The series expression \eqref{eq39m} for  $F_k(\zeta e^{-t})$ converges pointwise to a continuous function of $t$ for $t\geq 0$.
\end{Lemma}
Assuming Lemma~\ref{lem64} for a moment, continuity (from the right) at $t=0$ gives 
$$\lim_{t\newsmallsearrow 0}F_k(\zeta e^{-t}) = \operatorname{Im}\sum_{\ell,\rm odd} \frac{\zeta^{k\ell}}{\ell} = \frac 1 2 \arg\frac{1+\zeta^{k}}{1-\zeta^k}= \frac \pi 4 \epsilon_k~,$$ proving Lemma~\ref{lem63} and, hence, the proposition.  Here the second equality holds because by Abel's theorem on power series, the Taylor expansion $2\sum\limits_{\ell, \rm odd}\frac{x^\ell}{\ell}$ converges (non-absolutely) to  $\log\frac{1+x}{1-x}$ for every $x\neq 1$ on the unit circle, and the third equality follows from (\ref{eq26n}). \end{proof}
\begin{proof}[Proof of Lemma~\ref{lem64}.] We proceed by ``integration by parts'' as in the proof of Abel's theorem, as follows. 
Set $P(x)=\frac{2x^{k+2p}}{1+x^p}$ and $S_\ell=\operatorname{Im}\sum_{\ell'\le \ell,\text{ odd}}\zeta^{k\ell'}$ so that
for all $t\geq 0$
\begin{equation} \label{eq29c} F_k(\zeta e^{-t})=
\sum_{\ell, \mathrm{ odd}}(S_\ell-S_{\ell-2})\frac{P(e^{-t\ell})}{\ell}
=\sum_{\ell, \mathrm{ odd}}S_\ell\Bigl(\frac{P(e^{-\ell t})}{\ell}-\frac{P(e^{-(\ell+2)t})}{\ell+2}\Bigr)~.
\end{equation}
We write $f(u)=P(e^{-ut})/u$ so that 
$$\frac{P(e^{-\ell t})}{\ell}-\frac{P(e^{-(\ell+2)t})}{\ell+2}=\int_{\ell+2}^\ell f'(u)du=\int_{\ell+2}^\ell \frac{P'(e^{-tu})e^{-tu}(-tu)-P(e^{-tu})}{u^2}du$$
We now bound from above
$|S_\ell|$
by a constant,  $P(e^{-tu})$ and $P'(e^{-tu})$ by a constant, $e^{-tu}tu$ by a constant and $u^2$ by $\ell^2$ from below (here $t\geq 0, u\geq 1$.) This shows that for all $t\geq 0$,
the general term of the series (\ref{eq29c})
is bounded in absolute value by $C/\ell^2$ for some constant $C$ independent of $t$. By dominated convergence, the result follows.
 \end{proof}

\subsubsection{Transformation law for $S(\frac{q}{p})$}

From its very definition $S$ satisfies the equation
$S(\frac{q+2p}{p})=S(\frac{q}{p})$. It furthermore satisfies the following:

\begin{Proposition}\label{prop67}
For any two coprime positive integers $p,q$ with $q$ odd,
one has:
  \begin{equation}\label{eq29}
S\Bigl(\frac{q}{q+p}\Bigr)-S\Bigl(\frac{q}{p}\Bigr)=\frac 1 2~.
    \end{equation}
 \end{Proposition}

\begin{Remark}{\em 
Note that this implies the identity $S(\frac{q}{2q+p})-S(\frac{q}{p})=1$ mentioned in the introduction.
}\end{Remark}
\begin{proof}
  The proposition can be deduced from Proposition~\ref{dedekindtosign} and the reciprocity formula of Dedekind
  sums $s(q,p)+s(p,q)=
  \frac{1}{12}\big(\frac{q}{p}+\frac{p}{q}+\frac{1}{pq}\big)-\frac{1}{4}$
  by a direct computation.
In the context of this 
paper, it is however natural to derive this transformation law using only 
Proposition~\ref{Smodular} which expresses $S(\frac q p)$ as a boundary value of the function $\arg g$.
Let us explain how this goes.

The point is that the function $g(\tau)= \eta(\tau)^2/\eta(2\tau)$  is modular\footnote{Modular in a generalized sense, as the transformation law of $g(\tau)$ for most $\phi\in\Gamma_0(2)$ involves a non-trivial eigth root of unity depending on $\phi$, {\em cf.} formula (\ref{eq31a}).} of weight $\frac 1 2$ for the level $2$ subgroup
$$\Gamma_0(2) = \{\begin{pmatrix} a & b \\ c & d \end{pmatrix}\in \SL_2(\Z)\  |\  c\equiv 0 \pmod 2\}$$ generated by the two transformations $\phi_1(\tau)=\tau+1$ and $\phi_2(\tau)=\frac{\tau}{2\tau+1}$. This means more precisely that one has  $g(\tau+1)=g(\tau)$ (which is obvious) and 
\begin{equation}\label{eq31a}
      g(\phi_2(\tau))=g\Bigl(\frac \tau {2\tau+1}\Bigr)= \sqrt {-i(2\tau+1)}\, g(\tau)~,
\end{equation}
 where we use the principal  branch of the square
    root for which $\sqrt{1}=1$. The transformation law (\ref{eq31a}) follows from the  identity
\begin{equation}\label{eq30}
g(\tau)=\theta(2\tau -1) 
\end{equation} where $\theta(\tau)$ is the classical theta function defined on the upper half
plane $\H$ by $\theta(\tau)=\sum_{n\in\Z}e^{\pi i n^2 \tau}$. Indeed, it is well-known that the theta function satisfies the  transformation law 
\begin{equation}\nonumber \theta\Bigl(-\frac 1 \tau\Bigr) = \sqrt {-i\tau}\, \theta(\tau)
    \end{equation} and this implies (\ref{eq31a}), using $\theta(\tau +2)=\theta(\tau)$ 
  in the last step :  
\begin{eqnarray}
       g\Bigl(\frac \tau {2\tau+1}\Bigr)&=& \theta\Bigl(\frac {2\tau}{2\tau +1} -1\Bigr) 
  =  \theta\Bigl(\frac {-1}{2\tau+1}\Bigr)=\sqrt {-i(2\tau+1)}\, \theta(2\tau+1) \nonumber\\
        &=& \sqrt {-i(2\tau+1)}\, g(\tau)~.\nonumber
 \end{eqnarray} The identity (\ref{eq30}) is well-known.
It can for example be deduced from the Jacobi triple product identity, 
  see Eq.~(1.5) in \cite[\S 1.1]{K}.

    The rational numbers equivalent to $\frac 1 0 = i\infty$ under the action of $\Gamma_0(2)$ are precisely the numbers $\frac q {2p}$  where $q,p$ are coprime integers with $q$ odd. The transformation $\phi_2(\tau)=\frac{\tau}{2\tau+1}$ sends $\frac q {2p}$ to  $\frac q {2(q+p)}$. The transformation law (\ref{eq31a}) gives a transformation law for $\arg g$ of the form 
\begin{equation}\label{trlw}
      \arg g(\phi_2(\tau)) - \arg g(\tau) = \arg \sqrt {-i(2\tau+1)} + 2\pi N,
    \end{equation} where $N$ is some integer independent of $\tau$. (We will see later that $N=0$.)
    Applying this  to the vertical line $\tau=\frac q {2p}+it$ and taking the limit as $t\newsearrow 0$, Proposition~\ref{Smodular} gives that $\arg g(\frac q {2p}+it)$ goes to $- \frac {\pi} 2 S(\frac q {p})$, and one can check
\footnote{
This also follows from Proposition~\ref{Smodular}, because the limit as $t\newsearrow 0$ of $\arg g$ on the half-circle  $\phi_2(\frac q {2p}+it)$ is the same as its limit on the vertical line $\frac q {2(q+p)} +it$. One can give an elementary proof of this fact by  transporting $\frac q {2(q+p)}$ to $\frac 1 0 = i\infty$ by some element of $\Gamma_0(2)$ and analyzing the function $\arg g$ there. }
that $\arg g(\phi_2(\frac q {2p}+it))$ goes to $ -\frac {\pi}2  S(\frac q {q+p})$. Furthermore, $\arg \sqrt {-i(2\tau+1)}$ goes to $\arg \sqrt{-i(\frac q p +1)}=- \frac {\pi} 4 $. (The sign is $-1$ here because $q$ and $p$ are positive.)  Thus  the transformation law \eqref{trlw} for $\arg g$ gives 
\begin{equation}\nonumber 
S\Bigl(\frac q {q+p}\Bigr) - S\Bigl(\frac q p\Bigr) = - \frac 2 \pi \Bigl(-\frac \pi 4 +2\pi N\Bigr) = \frac 1 2 -4N
\end{equation}
It remains to see that $N=0$, but since $N$ is a constant independent of $p$ and $q$, we can just compute $N$ in an example. 
Since $S(\frac{1}{p})=\frac{p-1}{2}$, we have $S(\frac{1}{p+1})-S(\frac 1 p)=\frac{1}{2}$ hence $N=0$. This completes the proof.
\end{proof}


\subsection{A  modular form related to the genus 2 signature}\label{sectiong2}

For $k\in \Z$ and $n\in \Z_{>0}$, define $\sigma_k(n)=\sum\limits_{d|n}d^k$. Given  $\tau$ in $\H$ and $z=e^{2i\pi\tau}$, we set 
$$E_4^{\rm odd}(\tau)=\sum_{n\ge 1,  \rm odd} \sigma_3(n)z^{n/2}.$$
 
Note that this is a variant of the Eisenstein series $E_4$. We take from  \cite[Theorem~5.8]{Stein} that $E_4^{\rm odd}(\tau)$ is a weight 4 modular form for $\Gamma(2)$. \footnote{More precisely, \cite[Theorem~5.8]{Stein} proves the equivalent statement that $E_4^{\rm odd}(2\tau)$  is a weight $4$ modular form for  $\Gamma_0(4)$.}

Let $G(\tau)$ be the following Eichler integral associated to $E_4^{\rm odd}$.
This  means that $G(\tau)$ is  a primitive of order 3 of $E_4^{\rm odd}(\tau)$ with respect to the variable $\tau$.
(See \cite[Appendix A.1]{BKSZ}.)
As $\frac{d}{d\tau}z^{n/2}=\frac{d}{d\tau} e^{i\pi n \tau}=  i\pi nz^{n/2}$, we can choose the following:
$$G(\tau)=\frac{1}{(i\pi)^3}\sum_{n\ge 1, \mathrm{ odd}}\frac{\sigma_3(n)z^{n/2}}{n^3}=\frac{1}{(i\pi)^3}\sum_{n\ge 1, \mathrm{ odd}}\sigma_{-3}(n)z^{n/2}.$$

Now recall from Theorem~\ref{asympto}
the function
$$\Lambda(\theta)=\frac{16}{\pi^3}\sum_{n\ge 1, \textrm{ odd}}\frac{1}{n^3\sin(n\pi\theta)}~.$$
The following proposition relates certain boundary values of the Eichler integral $G(\tau)$ to the function $\Lambda(\theta)$. Recall that
we checked in Section~\ref{sec2} that
$\Lambda(\theta)$ converges for $\theta$ irrational with $\mu(\theta)<3$. It clearly also converges for $\theta$ rational with even denominator.

\begin{Proposition}\label{6.5n} For any irrational $\theta$ satisfying $\mu(\theta)<3$ or rational $\theta$  with even denominator one has:
$$\lim_{\tau\to \theta}G(\tau)=-\frac{1}{32}\Lambda(\theta)~.$$
\end{Proposition}
\begin{proof} The following computation shows that $G(\tau)$ when viewed as a power series in $z^{1/2}= e^{i\pi\tau}$ has convergence radius equal to $1$, and that it also converges for $z^{1/2}=e^{i \pi \theta}$ on the boundary of the disk of convergence when $\theta$ is irrational with $\mu(\theta)<3$ or is rational with even denominator, as  it converges to $-\frac{1}{32}\Lambda(\theta)$ there. Thus the result follows from Abel's theorem
  for power series.

\begin{eqnarray*}
G(\tau)&=&\frac{1}{(i\pi)^3}\sum_{N\geq 1,\textrm { odd}}\Bigl(\sum_{n|N}\frac 1 {n^3} \Bigr)z^{N/2}\\
&=&\frac{1}{(i\pi)^3}\sum_{n\geq 1,\textrm{ odd}}\frac{1}{n^3} \sum_{k=0}^\infty z^{\frac{n(2k+1)}{2}}=\frac{1}{(i\pi)^3}\sum_{n\geq 1,\textrm{ odd}}\frac{z^{n/2}}{n^3(1-z^n)}\\
       &=& \frac{i/2}{(i\pi)^3}\sum_{n\ge 1, \textrm{ odd}}\frac{1}{n^3\sin(n\pi\tau)}
           = -\frac{1}{2 \pi^3}\sum_{n\ge 1, \textrm{ odd}} \frac{1}{n^3\sin(n\pi\tau)}~.
           \end{eqnarray*}
         \end{proof}

    The Eichler integral $G(\tau)$ is not a modular form, but  thanks to Bol's identity \cite{Bo} (see also \cite[Prop.~1 in Appendix A.1.3]{BKSZ}), it satisfies the following: for any $\phi\in\Gamma(2)$, there exists a polynomial $P_\phi\in \C[\tau]$ of degree 2 (called  period polynomial) such that 
         \begin{equation}\label{Gmod}
G(\tau)-G(\phi(\tau))(c\tau+d)^2=P_{\phi}(\tau),\quad \textrm{ where }\phi(\tau)=\frac{a\tau+b}{c\tau+d}.
\end{equation}
In other words, $G(\tau)$ behaves like a modular form of weight $-2$ for $\Gamma(2)$, up to error terms given by the period polynomials $P_\phi$.

Recall that $\Gamma(2)$ is generated by  $\phi_1(\tau)=\tau+2$ and $\phi_2(\tau)=\frac{\tau}{2\tau+1}$. By the definition of $G(\tau)$,  one obviously has $G(\tau+2)=G(\tau)$, so $P_{\phi_1}=0$. In the following proposition we compute the polynomial $P_{\phi_2}$. This will allow us to guess a (conjectural) transformation law for $\sigma_2(\frac q {2q+p})$ in terms of $\sigma_2(\frac q p )$ analogous to the transformation law (\ref{eq29}) for $S(\frac q p)$.

\begin{Proposition}\label{6.6n}
For $\phi_2(\tau)=\frac{\tau}{2\tau+1}$ one has $P_{\phi_2}(\tau)=\frac 1 {32} (2\tau^2+2\tau+1)$.
\end{Proposition}
\begin{proof}
We first observe that  by taking limits $\tau \to\theta$  in (\ref{Gmod}) we get from Prop.~\ref{6.5n} the following functional identity from which we deduce that $P_{\phi_2}$ has real coefficients:
\begin{equation}\label{Lambdamod}
\Lambda\Bigl(\frac{\theta}{2\theta+1}\Bigr)(2\theta+1)^2-\Lambda(\theta)=32\,   P_{\phi_2}(\theta)
\end{equation}

Let $\sigma=\frac{-1+i}{2}$ so that
$\frac{\sigma}{2\sigma+1}=\sigma+1$ and 
$2\sigma^2+2\sigma+1=0$.
Taking $\tau=\sigma$ in (\ref{Gmod}), as $G(\frac{\sigma}{2\sigma+1})=G(\sigma+1)=-G(\sigma)$ and $(2\sigma+1)^2=-1$, we find $P_{\phi_2}(\sigma)=0$. As $P_{\phi_2}$ has real coefficients, we also have $P_{\phi_2}(\bar\sigma)=0$. Since we know that 
$P_{\phi_2}$ has degree $2$, this shows that  $$P_{\phi_2}(\tau)=C(2\tau^2+2\tau+1)$$ for some constant $C\in \R$.

Consider now $\tau=it$ for $t>0$ in (\ref{Gmod}). We get
\begin{equation}\label{eq38c} G(it)  - G\Bigl(\frac {it} {2it+1}\Bigr)(2it+1)^2 = P_{\phi_2}(it) = C(-2t^2 +2it +1)~.
  \end{equation}
Letting $t\to \infty$, clearly $G(it)$ goes to $0$ and $G(\frac{it}{2it+1})$ goes to $G(\frac 1 2)=-\frac{1}{32}\Lambda(\frac{1}{2})$ by Prop.~\ref{6.5n}.  One has $\Lambda(\frac{1}{2})=\frac{1}{2}$ thanks to the special value
 $$\sum_{n=0}^{\infty}\frac{(-1)^n}{(2n+1)^3}=\beta(3)= \frac{\pi^3}{32},$$ where $\beta(s)$ is the Dirichlet beta function (see {\em e.g.} \cite{D}). 
 Thus (\ref{eq38c}) gives 
$$ - (-\frac 1 {32})\frac 1 2 (-4t^2) \sim C (-2 t^2)$$ as $t\to \infty$,
hence $C=\frac 1{32}$, as asserted. 

\end{proof}

\begin{Corollary} For $\theta$ irrational with $\mu(\theta)<3$ or rational with even denominator, the function $\Lambda(\theta)$ satisfies
  the transformation laws $\Lambda(\theta+1)= -\Lambda(\theta)$ and
  \begin{equation}\label{eq39n} \Lambda\Bigl(\frac{\theta}{2\theta+1}\Bigr)(2\theta+1)^2-\Lambda(\theta)= 2\theta^2+2\theta+1~.\end{equation}
\end{Corollary}

\begin{proof} The first transformation law is immediate from the definition of  $\Lambda(\theta)$. The second one follows from Proposition~\ref{6.6n} and Eq.~(\ref{Lambdamod}).
\end{proof}

As already mentioned, our original motivation for the above computation was to find a transformation law for the genus two signature $\sigma_2(\frac q p)$ in the spirit of the transformation law (\ref{eq29}) for the genus zero signature $S(\frac q p)$. Indeed, in view of Theorem \ref{asympto} which for generic 
irrational $\theta\in[0,1]$
 exhibits $\Lambda(\theta)$ as a limit of  
$\sigma_2(\frac {q_k}{p_k})/p_k^2$ when $\frac {q_k}{p_k}$ is the sequence of convergents of $\theta$,  it is natural to expect that the transformation law \eqref{eq39n} lifts to an arithmetic avatar in terms of signatures. We found in this way the following conjectural statement that we checked by computer for all $0<q<p<100$. 

\begin{Conjecture}\label{6.8}
For any coprime odd integers $0<q<p$, we have 
\begin{equation}\label{eq39b}
  \sigma_2\Bigl(\frac{q}{2q+p}\Bigr)-\sigma_2\Bigl(\frac{q}{p}\Bigr)=2q^2+2pq+p^2-1.
  \end{equation}
\end{Conjecture}
\begin{Remark}
{\em The conjecture holds when $q=1$ by a direct computation using that $\sigma_2(\frac 1 p)=\dim \boV_p(S_2)=\binom{p+1}{3}$.}
\end{Remark}
\begin{Remark}{\em
    Conjecture~\ref{6.8} implies the transformation law (\ref{eq39n}) for $\Lambda(\theta)$  for almost all irrational $\theta$: just divide \eqref{eq39b} by $p^2$ and apply it to the sequence of convergents of $\theta$, then apply  Theorem~\ref{asympto}. But the converse is
    not true.}
\end{Remark}

In early February 2026, we were informed by Y. Murakami \cite{YMu} that he has obtained a proof of Conjecture~\ref{6.8}  starting from our trigonometric formula in Theorem~\ref{sigma2}. A nice feature of his proof is that he shows that the difference between $\frac{1}{p^2}\sigma_2(q/p)$ and its asymptotic limit is also the boundary value of an Eichler integral, this time of a modular form of weight 2.
Still, it would be very interesting to find a direct relation between $\sigma_2(\frac{q}{p})$ defined as a sum of signs as in Formula \eqref{defsigma2} and boundary values of modular forms (or their Eichler integrals) in the spirit of Section \ref{SectionS}.
It looks also challenging to us to extend these results to higher genus and include marked points in the statements. If $\sigma_g(\frac{q}{p})$ indeed grows as $p^{2g-2}$ for any $g\ge 2$ as in
Question~\ref{51},
  one might hope that its asymptotics is governed by an Eichler integral of a modular form of weight $2g$ as in the genus $2$ case. 

\section{Appendix}
The following code in Sage computes $\sigma_g(\frac{q}{p})=S(p,q,g)=\Tr_{V_{\frac q p}/\Q}\Omega^{g-1}$. It is very efficient and relies on a formula for $\Omega\in V_{q/p}$ proven in \cite{S2B}.

\tiny{
\begin{verbatim}
R.<x> = QQ[]
# The following code computes the trace of the multiplication by P(x) on the ring QQ[x]/Q(x).
def Trace(P,Q): 
    n=Q.degree()
    s=0
    for k in range(n):
        U=(x^k*P).quo_rem(Q)[1]+x^(n+1)
        s=s+(U.list())[k]
    return(s)
# This code computes the signature function
def S(p,q,g):
    M=Matrix(p-1,p-1)
    for i in range(p-2):
        M[i+1,i]=1
        M[i,i+1]=(-1)^(1+int((i+1)*q/p)+int((i+2)*q/p))
    P=M.charpoly()
    Pprime=diff(P,x)
    iota=M[0:p-2,0:p-2].charpoly()
    Om=((-iota*Pprime)^(g-1)).quo_rem(P)[1]
    return(Trace(Om,P))
\end{verbatim}}

\end{document}